\documentclass[11pt,reqno]{amsart}

\usepackage{graphicx}
\usepackage{mathrsfs}
\usepackage{color}
\usepackage{comment}
\usepackage{amssymb}

\usepackage[margin = 1.4in] {geometry}

\usepackage[hyperindex,breaklinks]{hyperref}

\newtheorem{prop}{Proposition}
\newtheorem{theo}[prop]{Theorem}
\newtheorem*{theo*}{Theorem}
\newtheorem{lemm}[prop]{Lemma}
\newtheorem{coro}[prop]{Corollary}

\newtheorem*{claim}{Claim}

\theoremstyle{definition}
\newtheorem{rema}[prop]{Remark}
\newtheorem{defi}[prop]{Definition}

\newcommand{\NN}{\mathbf{N}}

\newcommand{\RR}{\mathbf{R}}

\newcommand{\cA}{\mathcal A}
\newcommand{\cB}{\mathcal B}

\newcommand{\cH}{\mathcal H}

\newcommand{\cR}{\mathcal R}
\newcommand{\cS}{\mathcal S}

\def\fA{\mathfrak{A}}
\def\fB{\mathfrak{B}}

\DeclareMathOperator{\supp}{supp}

\DeclareMathOperator{\Vol}{Vol}

\DeclareMathOperator{\ind}{index}

\DeclareMathOperator{\Ric}{Ric}

\DeclareMathOperator{\Div}{div}

\newcommand{\bangle}[1]{\left\langle #1 \right\rangle}

\newcommand{\eps}{\varepsilon}

\begin{document}

\title{Stable minimal hypersurfaces in $\mathbf{R}^{4}$}
\author{Otis Chodosh}
\address{Department of Mathematics, Stanford University, Building 380, Stanford, CA 94305, USA}
\email{ochodosh@stanford.edu}
\author{Chao Li}
\address{Courant Institute, New York University, 251 Mercer St, New York, NY 10012, USA}
\email{chaoli@nyu.edu}

\begin{abstract}
We prove that a complete, two-sided, stable minimal immersed hypersurface in $\mathbf{R}^{4}$ is flat. 
\end{abstract}

\maketitle

\section{Introduction}
A complete, two-sided, immersed minimal hypersurface $M^{n}\to\RR^{n+1}$ is \emph{stable} if
\begin{equation}\label{eq:stable}
\int_M |A_M|^2 f^2 \leq \int_M |\nabla f|^2
\end{equation}
for any $f \in C^{\infty}_{0}(M)$. We prove here the following result.

\begin{theo}\label{theo:bernstein}
A complete, connected, two-sided, stable minimal immersion $M^{3}\to\RR^{4}$ is a flat $\RR^{3}\subset \RR^{4}$. 
\end{theo}

This resolves a well-known conjecture of Schoen (cf.\ \cite[Conjecture 2.12]{CM:book}). The corresponding result for $M^{2}\to\RR^{3}$ was proven by Fischer-Colbrie--Schoen, do Carmo--Peng, and Pogorelov \cite{fischer-colbrie-schoen,docarmo-peng,pogorelov} in 1979. Theorem \ref{theo:bernstein} (and higher dimensional analogues) has been established under natural cubic volume growth assumptions by Schoen--Simon--Yau \cite{SSY} (see also \cite{Simons,SchoenSimon}). Furthermore, in the special case that $M^{n}\subset \RR^{n+1}$ is a minimal graph (implying \eqref{eq:stable} and volume growth bounds) flatness of $M$ is known as the \emph{Bernstein problem}, see \cite{Flemming,DG:bernstein,Almgren,Simons,BDG}. Several authors have studied Theorem \ref{theo:bernstein} under some extra hypothesis, see e.g., \cite{SY:harmonic.stable.minimal,CaoShenZhu,Berard,ShenZhu,Chen:curvature,NelliSoret,Miyaoka,Palmer:L2,Tanno}. We also note here some recent papers \cite{CFRS,FigalliSerra}  concerning stability in related contexts.

It is well-known (cf.\ \cite[Lecture 3]{White:notes}) that a result along the lines of Theorem \ref{theo:bernstein} yields curvature estimates for minimal hypersurfaces in $\RR^4$.
\begin{theo}\label{theo:curvature}
	There exists $C<\infty$ such that if $M^3\rightarrow \RR^4$ is a two-sided, stable minimal immersion, then
	\[
	|A_M (p)| d_M(p,\partial M) \le C.
	\]
\end{theo}
More generally, we recall that a minimal immersion $M^{3}\to (N^{4},g)$ is stable if 
\[
\int_M (|A_M|^2+\Ric_{N}(\nu,\nu)) f^2 \leq \int_M |\nabla f|^2
\]
for all $f \in C^{\infty}_{0}(M\setminus\partial M)$. 
\begin{theo}\label{theo:curvature.Riem}
	Suppose $(N^4,g)$ is a closed Riemannian manifold. There exists $C=C(N,g)$ such that if $M^3\rightarrow N^4$ is a two-sided, stable minimal immersion, then
	\[
	|A_M (p)| \min\{ 1, d_M(p,\partial M)\} \le C.
	\]
\end{theo}
We have recently generalized Theorem \ref{theo:curvature.Riem} to hold in (non-compact) ambient $(N^4,g)$ with bounded sectional curvature in a joint work with Stryker \cite[Corollary 2.5]{ChodoshLiStryker2022stableminimal}, which resolves \cite[Conjecture 2.13]{CM:book}. 

Theorems \ref{theo:curvature} and \ref{theo:curvature.Riem} are the four-dimensional analogue of the well-known curvature estimate of Schoen \cite{Schoen:estimates} for minimal surfaces in three-dimensions. Note that by the work of Schoen--Simon--Yau \cite{SSY}, such an estimate was previously known to hold where $C$ depended on an upper bound for volume of $M^3$ in small balls. 

\begin{rema}
There have been several interesting developments since the first version of this paper was posted. The authors have discovered \cite{CL:aniso} a new proof of Theorem \ref{theo:bernstein} that can be localized to obtain (interior) volume estimates (in the spirit of Pogorelov \cite{pogorelov}; cf.\ \cite{CM:elliptic}). This new proof is related to the study of \emph{uniformly positive} scalar curvature, while the current paper is related to the study of \emph{non-negative} scalar curvature. Subsequently, Catino--Mastrolia--Roncoroni have discovered \cite{CMR} a completely different proof of Theorem \ref{theo:bernstein}, related to the study of Bakry--\'Emery--Ricci curvature. Interestingly, the dimension restriction $n+1=4$ enters each proof in a different way. 
\end{rema}

A slight modification of the proof of Theorem \ref{theo:bernstein} yields a structure theorem for finite index minimal hypersurfaces in $\RR^4$, analogous to the well-known results of Gulliver, Fischer--Colbrie, and Osserman \cite{FC,Gulliver,Osserman}. Recall that a complete, two-sided, immersed minimal hypersurface $M^n\to\RR^{n+1}$ has \emph{finite Morse index} if 
\[
\ind(M) : = \sup\left\{\dim V : V\subset C^\infty_0(M), Q(f,f) < 0 \textrm{ for all } f\in V\setminus\{0\}\right\} < \infty,
\]
where $Q (f,f)=  \int_M |\nabla f|^2 - \int_M |A_M|^2 f^2$. 

\begin{theo}\label{theo:index}
A complete, two-sided, minimal immersion $M^3\to\RR^4$ has finite Morse index if and only if it has finite total curvature $\int_M |A_M|^3< \infty$. 
\end{theo}

We remark that Tysk \cite{Tysk} proved the same statement for a complete, two-sided, minimal immersion $M^n\to \RR^{n+1}$ (for $3\leq n\leq 6$) under the assumption that $M$ has Euclidean volume growth. 

Theorem \ref{theo:index} has strong consequences on the structure of $M$ near infinity. We recall the following definition.

\begin{defi}[{\cite[Section 2]{Schoen:symmetry}}]
	Suppose $n\ge 3$, $M^n \to \RR^{n+1}$ is a complete minimal immersion. An end $E$ of $M$ is \emph{regular at infinity} if it is the graph of a function $w$ over a hyperplane $\Pi$ with the asymptotics
	\[w(x)=b+a|x|^{2-n} + \sum_{j=1}^n c_j x_j |x|^{-n}+O\left(|x|^{-n}\right),\]
	for some constants $a,b,c_j$, where $x_1,\cdots,x_n$ are the coordinates in $\Pi$.
\end{defi}

By \cite{Anderson,Schoen:symmetry,ShenZhu,Tysk} (see also \cite[Appendix A]{Li:index4d}), if $M^n \to \RR^{n+1}$ is a minimal immersion with finite total curvature, then each end of $M$ is regular at infinity. Moreover, by \cite{LiWang} a two-sided minimal immersion with finite Morse index has finitely many ends. Combined with Theorem \ref{theo:index}, this yields the following result. 

\begin{coro}\label{coro:index}
	Suppose $M^3\to \RR^4$ is a complete, two-sided, minimal immersion with finite Morse index. Then $M$ has finitely many ends, each of which is regular at infinity. In particular, $M$ has cubic volume growth, i.e.
	\[\sup_{r\to \infty} \frac{\Vol\left(B_r(0)\cap M\right)}{r^3}<\infty.\]
\end{coro}

In fact, by \cite{Li:index4d}, we can bound the volume growth rate linearly in terms of the Morse index
\[\sup_{r\to \infty} \frac{\Vol\left(B_r(0)\cap M\right)}{\frac{4}{3} \pi r^3}\le 6\left(\ind(M)+1\right).\]
These results should be relevant to the study of minimal hypersurfaces with bounded index in a four-manifold (cf.\ \cite{CKM,Song:index}).

\begin{rema}
We also note that Theorems \ref{theo:bernstein} and \ref{theo:curvature} hold if \eqref{eq:stable} is weakened to 
\[
(1-\delta_s) \int_M |A_M|^2 f^2 \leq \int_M |\nabla f|^2
\]
for a constant $\delta_s\in [0,\tfrac 13]$. This notion (called $\delta_s$-stability) was introduced by Colding--Minicozzi \cite{CM:disks2} (cf.\ \cite{CM:elliptic}) as a technical tool to quantify the well-known phenomenon that smooth convergence with multiplicity yields stable minimal hypersurfaces. (Note that Tam-Zhou \cite{TamZhou2009stability} proved that the $3$-dimensional catenoid is $\tfrac 23$-stable.) We discuss this further in Appendix \ref{app:almost.stable}.
\end{rema}

\subsection{Idea of the proof of Theorem \ref{theo:bernstein}}
Consider a non-flat, complete, two-sided stable minimal hypersurface $M^{3}\to\RR^{4}$. It turns out to be no loss of generality to assume that $M$ is simply connected and of bounded curvature. 

It is well-known (see, e.g. \cite{Li:harmonic.lectures}) that $M$ admits a positive Green's function of the Laplacian operator denoted here by $u$. (In other words, $M$ is non-parabolic.) The central quantity considered in the proof of Theorem \ref{theo:bernstein} is
\[
F(t) := \int_{\Sigma_{t}} |\nabla u|^{2}
\]
where $\Sigma_{t} = u^{-1}(t)$. (Quantities of this type have been considered in several contexts including \cite{Ni:mean.value,Colding2012monotonicity,CM:pnas,CM:family.mon,MunteanuWang,AFM:sharp.Ric,AMO:pmt}.) A simple computation shows that $F(t) = 4\pi t^2$ when $M = \RR^{3}$. 

For general $M$, we can view any estimate of $F(t)$ as $t\to0$ as certain kind of control on the growth rate of $M$ near infinity. It is straightforward to show that $F(t)$ satisfies the \emph{a priori} estimate $F(t) = O(t)$ as $t\to 0$ (see Lemma \ref{lemm:F.apriori.bds}). An important observation is that if we can upgrade this to the sharp estimate $F(t) = O(t^{2})$, then we can conclude that $M$ is flat. 

At a heuristic level, if we were to pretend that $u\sim r^{-1}$ and $|\nabla u| \sim r^{-2}$ (this holds on $\RR^{3}$), we would have $F(t) \sim t^{4}|\partial B_{t^{-1}}| = r^{-4}|\partial B_{r}|$. Thus $F(t) = O(t^{2})$ can be thought of as the estimate $|\partial B_{r}| = O(r^{2})$ which would suffice to show $M$ is flat by work of Schoen--Simon--Yau \cite{SSY}. In fact, a closely related argument actually works: we can consider $f = \varphi \circ u$ in the $L^{3}$-version of the Schoen--Simon--Yau estimates (see Proposition \ref{prop:SSY.L3}) and use the co-area formula combined with a log-cutoff near the pole and the end to see that $F(t) = O(t^2)$ implies that $M$ is flat. 

It thus remains to explain how to show $F(t) = O(t^{2})$. 
Our strategy is to find two (competing) estimates relating $F(t)$ and the quantity
\[
\cA(t) = \int_{\Sigma_{t}} |A_{M}|^{2}.
\]
One can view $\cA(t)$ as measuring the bending of $M$ near infinity in a certain sense. 

Our first estimate follows by extending a recent
\footnote{See also their followup work \cite{MW:geometry-scalar} which appeared after the first version of this article was posted.}  
argument of Munteanu--Wang who obtained \cite{MunteanuWang} a sharp monotonicity formula for the quantity $F(t)$ on a $3$-manifold with non-negative scalar curvature by a clever application of Stern's rearrangement of the Bochner formula \cite{Stern} (allowing them to leverage Gauss--Bonnet on the level sets $\Sigma_t$). 

In the present situation, $M$ actually has non-positive scalar curvature $R_{M} = -|A_{M}|^{2}$, so this comes in with a bad sign when applying their method. Moreover, our \emph{a priori} bound $F(t)=O(t)$ causes complications with a crucial step in their argument. However, keeping track of this bad sign and developing a delicate regularization procedure to control this limt appropriately (cf.\ Remark \ref{rema:MW.ric}), we find 
\begin{align*}
F(t)  
& \leq  O(t^{2})  + \frac 14 t \int_0^t   \cA(s)\, ds  + \frac 14 t^{3}   \int_t^1 s^{-2} \cA(s)\, ds
\end{align*}
(see Proposition \ref{prop:MW-monotonicity}). To obtain an estimate for $F(t)$ we thus need to estimate $\cA(t)$ appropriately. We achieve this by considering\footnote{This part of the argument is inspired by the work of Schoen--Yau \cite{SY:harmonic.stable.minimal} who considered $f= |\nabla u|\psi$ and used the standard Bochner formula in \eqref{eq:stable} to prove that $M$ does not admit nonconstant finite energy harmonic functions.} the stability inequality \eqref{eq:stable} with $f = |\nabla u|^{\frac 12} \varphi(u)$. Using Stern's rearrangement of the Bochner formula and the co-area formula, we obtain
\begin{align*}
& \int_{0}^{\infty} \varphi(s)^{2}  \cA(s)\, ds  \leq \frac{8\pi}{3} \int_{0}^{\infty} \varphi(s)^{2} \, ds  +  \frac 43 \int_{0}^{\infty} \varphi'(s)^{2} F(s) \, ds,
\end{align*}
which yields
\begin{align*}
t \int_{0}^{t}   \cA(s)\, ds + t^3 \int_{t}^{1}  s^{-2}  \cA(s)\, ds   \leq O(t^{2})  +  \frac 43 t^3 \int_{t}^{1} s^{-4} F(s) \, ds
\end{align*}
after a judicious choice of the test function (see Proposition \ref{prop:stern.bochner.stable} and Corollary \ref{coro:cutoff}). 

Putting these inequalities together we obtain 
\[
F(t)  \leq  O(t^2)  +  \frac 13 t^3 \int_{t}^{1}  s^{-4} F(s) \, ds.
\]
Because the constant $\frac 13$ is $<1$, it is possible to ``absorb'' the integral into the left-hand side, yielding $F(t) = O(t^{2})$. As explained above, this implies that $M$ is flat. 

We emphasize that both of the competing inequalities described above rely on Stern's Bochner formula, and thus use the three-dimensionality of $M$ in a strong way. In particular, Stern's version of the Bochner formula as used in the stability inequality preserves more of the factor in front of $\cA(s)$ than the ``standard'' Bochner formula would (basically, this is due to the use of Gauss--Bonnet). Extra loss at this step would make the constant $\frac 13$ larger and could cause issues with the ``absorption'' step. 
 
\subsection{Organization of the paper} In Section \ref{sec:Greens} we collect basic properties of the Green's function on a stable minimal hypersurface. In Section \ref{sec:Stern} we combine Stern's Bochner formula with the stability inequality and in Section \ref{sec:MW} we extend the work of Munteanu--Wang and combine this with the previous section to prove that $F(t) = O(t^{2})$. In Section \ref{sec:proofs} we use this to prove Theorems \ref{theo:bernstein}, \ref{theo:curvature}, and \ref{theo:curvature.Riem}. Appendix \ref{app:ends} recalls the one-ended (finite-ended) property of stable (finite index) minimal hypersurfaces, Appendix \ref{app:harnack.sobolev} recalls Yau's Harnack and Michael--Simon Sobolev inequality in this setting, Appendix \ref{app:harmonic} recalls properties of the nodal/critical set of a harmonic function, and Appendix \ref{app:L3} proves a $L^{3}$ version of the Schoen--Simon--Yau estimates. Appendix \ref{app:almost.stable} contains an extension of Theorem \ref{theo:bernstein} to almost stable minimal immersions.

\subsection{Acknowledgements}
We are grateful to Rick Schoen and Christos Mantoulidis for their interest and encouragement, to Alice Chang for a helpful discussion, and to the referees for some suggestions concerning the exposition. O.C. was supported by an NSF grant (DMS-2016403), a Terman Fellowship, and a Sloan Fellowship. C.L. was supported by an NSF grant (DMS-2005287).

\section{The Green's function on a stable minimal hypersurface}\label{sec:Greens}

It is well-known (cf.\ \cite[Theorem 10.1]{Li:harmonic.lectures}) that when $n>2$ a stable minimal hypersurface in $\RR^{n+1}$ is non-parabolic (i.e., admits a positive Green's function). We recall this fact below as well as establishing several useful facts about the behavior of the level sets of such a Green's function. 

For $n\ge 3$, recall that the Euclidean Green's function on $\RR^n$ (at a point $p$) is $c_n |x-p|^{2-n}$, for $c_n^{-1} = (n-2)\omega_{n-1}$ (so $c_3^{-1} = 4\pi$); here, $\omega_{n-1}$ denotes the $(n-1)$-volume of the unit sphere in $\RR^{n}$.

\begin{prop}\label{prop:Greens}
For $n>2$, suppose that $M^n \to \RR^{n+1}$ is a complete, connected, simply connected, two-sided, stable minimal immersion with uniformly bounded curvature $|A_M| \leq K$. Then, for $p\in M$, there exists $u \in C^\infty_\textnormal{loc}(M\setminus\{p\})$ with the following properties:
\begin{enumerate}
\item $u$ is harmonic, $\Delta u = 0$, on $M\setminus\{p\}$,
\item $u>0$ on $M\setminus\{p\}$ and $\inf  u = 0$,
\item $u(x) = c_n(1+o(1))d_M(x,p)^{2-n}$, $|\nabla u|(x) = (n-2)c_n(1+o(1)) d_M(x,p)^{1-n}$, as well as $\bangle{\nabla |\nabla u|,\tfrac{\nabla u}{|\nabla u|}} = (n-1)(n-2)c_n(1+o(1)) d_M(x,p)^{-n}$  as $x\to p$,
\item if $K$ is a compact set containing $p$, then $\int_{M\setminus K} |\nabla u|^2 < \infty$,
\item $u(x)\to0$ as $d_M(x,p)\to \infty$,
\item for all $s\in (0,\infty)$, the set $\Omega_{s} : = \{u \geq s\} \cup \{p\}$ is compact, and
\item if $s \in (0,\infty)$ is a regular value of $u$ then $\partial\Omega_{s} = u^{-1}(s) : = \Sigma_s$ is a closed connected hypersurface in $M$.
\end{enumerate}
\end{prop}
\begin{proof}
If $M$ is a flat $\RR^n$, then we can take $u$ to be the standard Green's function on $\RR^n$. As such, we assume below that $M$ is not flat. 

We construct a sequence of approximations to $u$ following \cite[\S 2]{Li:harmonic.lectures}. Choose an exhaustion $\Omega_1 \subset \Omega_2 \subset \dots \subset M$ with $\Omega_i$ a pre-compact open set, $p\in \Omega_i$, and $\partial\Omega_i$ smooth. There exists $u_i \in C^\infty_\textrm{loc}(\bar \Omega_i \setminus\{p\})$ so that 
\begin{enumerate}
\item [(1')] $\Delta u_i = 0$ on $\Omega_i\setminus \{p\}$,
\item [(2')] $u_i>0$ on $\Omega_i$ and $u_i = 0$ on $\partial \Omega_i$, and
\item [(3')]  $u_i(x) = c_n (1+o(1))d_M(x,p)^{2-n}$, $|\nabla u_i|(x) = (n-2)c_n(1+o(1)) d_M(x,p)^{1-n}$, as well as $\bangle{\nabla |\nabla u_i|,\tfrac{\nabla u_i}{|\nabla u_i|}} = (n-1)(n-2)c_n(1+o(1)) d_M(x,p)^{-n}$ as $x\to p$.
\end{enumerate}
This is simply the statement of existence of a Green's function on a compact manifold with boundary. 

Note that $ u_i(x) \leq u_j(x)$ for $i<j$ and $x \in\Omega_i\setminus\{p\}$. Indeed, by (2'), for any $\delta>0$, it holds that $ u_i(x) \leq (1+\delta)u_j(x)$ for $x \in \partial B_\eps(p)$ for all $\eps>0$ sufficiently small (depending on $\delta$). Since $0= u_i(x) < u_j(x)$ for $x \in \partial \Omega_i$, we thus find $u_i \leq (1+\delta)u_j $ on $\Omega_i \setminus B_\eps(p)$. Sending $\delta,\eps\to0$, the inequality $u_i(x) \leq u_j(x) $ follows. 

Set 
\[
\mu_i : = \max_{x \in \partial \Omega_1} u_i(x). 
\]
The maximum principle implies that
\begin{equation}\label{eq:greens.upper.lower.construct}
 u_1 \leq  u_i \leq u_1 + \mu_i 
\end{equation}
on $\Omega_1$. 
We have seen that $\{\mu_i\}_{i\in\NN}$ is increasing. 
\begin{claim}[{cf.\ \cite[Theorem 2.3]{Li:harmonic.lectures}}] The sequence $\{\mu_i\}_{i\in\NN}$ is bounded above.
\end{claim}
\begin{proof}[Proof of the claim]
Assume that $\mu_i\to\infty$. The Harnack inequality (cf. Proposition \ref{prop:diff.harnack}) and interior estimates implies that up to passing to a subsequence, $\mu_i^{-1} u_i$ converges in $C^\infty_\textrm{loc}(M\setminus\{p\})$ to some non-negative harmonic function $u \in C^\infty_\textrm{loc}(M\setminus\{p\})$. 

Because we have assumed that  $\mu_i\to\infty$, then we find that $0\leq u \leq 1$ on $\Omega_1$. On the other hand, the maximum principle (and Dirichlet boundary conditions for $u_i$) implies that
\[
\max_{\Omega_i \setminus \Omega_1} u_i = \mu_i, 
\]
so $u \leq 1$ on $M \setminus \Omega_1$. The maximum principle then implies that $u\equiv 1$ on $M\setminus\{p\}$.  

Choose $f_i \in C^{0,1}(M) \cap C^\infty(\bar \Omega_i\setminus \Omega_1)$ so that
\[\begin{cases}
\Delta f_i = 0 & \textrm{in } \Omega_i \setminus\Omega_1\\
f_i = 1 & \textrm{in } \Omega_1\\
f_i = 0 & \textrm{in } M\setminus \Omega_i. 
\end{cases}
\]
The maximum principle implies that $\mu_i^{-1}u_i \leq f$ on $\Omega_i\setminus \Omega_1$, so we see that $f_i\to 1$ in $C^0_\textrm{loc}(M) \cap C^\infty_\textrm{loc}(M\setminus \Omega_1)$. Taking $f_i$ in the stability inequality \eqref{eq:stable} we find (cf.\ \cite[Theorem 10.1]{Li:harmonic.lectures})
\[
\int_M |A_M|^2 f_i^2 \leq \int_{\Omega_i\setminus\Omega_1} |\nabla f_i|^2 = - \int_{\partial\Omega_1} D_\nu f_i
\]
for $\nu$ the outwards pointing unit normal to $\Omega_1$. Letting $i\to\infty$, since $f_i$ converges to $1$, we find that
\[
\int_M |A_M|^2  = 0,
\]
a contradiction since we assumed that $M$ was not flat. 
\end{proof}
Write $\mu_i\to\mu_\infty$. The claim, the Harnack inequality (cf. Proposition \ref{prop:diff.harnack}), and interior estimates allow us to pass to a subsequence so that $u_i$ converges to a harmonic function $u$ in $C^\infty_\textrm{loc}(M\setminus\{p\})$. Note that $u > 0$ by the maximum principle. 

By \eqref{eq:greens.upper.lower.construct} we have 
\[
u_1 \leq u \leq u_1 + \mu_\infty. 
\]
After shifting $u$ so that $\inf  u = 0$, we see that we have established (1)-(3). (The derivative estimates in (3) follow because $u-u_1$ is bounded near $p$ and thus extends across $p$.) 

We now establish property (4). Consider the solution to the following problem
\begin{equation}\label{eq:comparison.dirichlet.wi}
\begin{cases}
\Delta w_i = 0 & \textrm{in } \Omega_i\setminus \Omega_1\\
w_i = 0 & \textrm{on }\partial\Omega_i\\
w_i = u  & \textrm{on }\partial\Omega_1  \\
\end{cases}
\end{equation}
By the maximum principle (and since $u_i$ is increasing to $u$), $u_i \leq w_i \leq u$ on $\Omega_i \setminus \Omega_1$. Thus, we find that $w_i$ converges to $u$ in $C^\infty_\textrm{loc}(M\setminus \Omega_1)$. Furthermore, we have that
\[
\int_{\Omega_i} |\nabla w_i|^2 \leq \int_{\Omega_1} |\nabla w_1|^2 
\]
for $i\in\NN$. Passing this inequality to the limit, we have proven (4). 

We now consider property (5) (cf.\ \cite[Remark 2.4(d)]{Ni:mean.value}). We claim that for any (intrinsically) diverging sequence $\{x_i\}_{i\in\NN}\subset M$ it holds that $u(x_i) \to 0$.  Choose $\varphi \in C^\infty(M)$ so that $\varphi \equiv 0$ in $\Omega_2$ and $\varphi \equiv 1$ on $M \setminus \Omega_3$. Consider the function $\varphi w_i$ where $w_i$ is as in \eqref{eq:comparison.dirichlet.wi}. Note that $\varphi w_i \in C^{0,1}_c(M)$. By the Michael--Simon--Sobolev inequality (cf.\ Proposition \ref{prop:MSS}) we have
\[
\left(\int_M (\varphi w_i)^{\frac{2n}{n-2}} \right)^{\frac{n-2}{n}} \leq C \int_M \varphi^2 |\nabla w_i|^2 + w_i^2 |\nabla \varphi|^2.
\]
Because $w_i$ has uniformly bounded Dirichlet energy and $\nabla \varphi$ is compactly supported we can let $i\to\infty$ (recalling that $w_i$ converges to $u$ in $C^\infty_\textrm{loc}(M\setminus\Omega_1)$ to find
\[
\left(\int_M (\varphi u)^{\frac{2n}{n-2}} \right)^{\frac{n-2}{n}} < \infty. 
\]
In particular,
\begin{equation}\label{eq:sob.applied.harm.fct.infty}
\lim_{i\to\infty} \left(\int_{M\setminus \Omega_i} u^{\frac{2n}{n-2}} \right)^{\frac{n-2}{n}} = 0. 
\end{equation}
We now assume (for contradiction) that 
\[
\lim_{i\to\infty} \sup_{M\setminus\Omega_i} u = a \in (0,\infty). 
\]
If this holds, we can find a (intrinsically) diverging sequence $\{x_i\}_{i\in\NN}\subset M$ with $u(x_i) \to a$ as $i\to\infty$. Passing to a subsequence we can assume that 
\[
d_M(x_i,\Omega_i) \to \infty
\]
as $i\to\infty$. Because $M$ has uniformly bounded curvature, Schauder estimates yield bounds on all derivatives of curvature. Thus, we can pass $(M,g_M,x_i)$ to the limit $(\tilde M,\tilde g,\tilde x)$ in the pointed Cheeger--Gromov sense (where $g_M$ is the induced metric on $M$). Note that by the Harnack inequality (cf.\ Proposition \ref{prop:diff.harnack}) and interior estimates we can also pass the harmonic function $u$ to a pointed limit $\tilde u$ (passing to a further subsequence). Note that $\tilde u$ is harmonic, $\tilde u(\tilde x) = a$ and $\tilde u \leq a$. Thus, the maximum principle implies that $\tilde u \equiv a$. On the other hand, \eqref{eq:sob.applied.harm.fct.infty} implies that $\tilde u \equiv 0$. This is a contradiction, completing the proof of property (5). 

Properties (3) and (5) imply property (6).

We now consider property (7). For $s$ a regular value, $\Omega_s = \{u \geq s\}$ is a compact set with smooth boundary $\Sigma_s = u^{-1}(s)$ a closed (regular) hypersurface.
\begin{claim}
For $s \in (s_0,\infty)$ a regular value of $u$, $M\setminus \Omega_s$ has exactly one component and it is unbounded. 
\end{claim}
\begin{proof}[Proof of the claim]
Assume that $M\setminus \Omega_s$ has a bounded component $\Gamma$. Note that $0 < u < s$ on $\Gamma$ and $u = s$ on $\partial\Gamma$. Thus, $u$ attains its minimum on the interior of $\Gamma$. This is a contradiction since $u$ is not constant. This shows that $M\setminus \Omega_s$ is disconnected then $M$ has at least two ends. This would contradict the one-ended result of Cao--Shen--Zhu (cf.\ Proposition \ref{prop:CSZ}). 
\end{proof}

To finish the proof of (7), if $\Sigma_s$ were disconnected, then by the claim we can construct a loop with non-trivial algebraic intersection with one of the components of $\Sigma_s$ (concatenate a path connecting two of the components of $\Sigma_s$ inside of $\Omega_s$ with one outside of $\Omega_s$). This contradicts the assumed simple connectivity of $M$. 
\end{proof}

\subsection{Integration over level sets of the Green's function} Given $M^n\to\RR^{n+1}$  a complete, non-compact, connected, simply connected  two-sided stable minimal immersion with uniformly bounded curvature. For $p\in M$ fixed, consider the Greens' function $u \in C^\infty_\textnormal{loc}(M\setminus\{p\})$ constructed in Proposition \ref{prop:Greens}. Set $\cR$ the regular values of $u$ and $\cS$ the singular values of $u$. Recall that we have defined $\Sigma_s : = u^{-1}(s)$. Below we will use this notation even when $s\in \cS$ is a singular value. 
\begin{lemm}\label{lemm:reg.open.dense}
The set of regular values $\cR$ is open and dense in $(0,\infty)$. 
\end{lemm}
\begin{proof}
Proposition \ref{prop:Greens} implies that $u : M\setminus\{p\} \to (0,\infty)$ is proper. This easily is seen to imply openness of $\cR$. Density follows from Sard's theorem as usual. 
\end{proof}

For any $s \in (0,\infty)$ we set
\[
\Sigma_s^* = \{x\in M : u(x) = s, |\nabla u|(x) > 0\}. 
\]
Note that $\Sigma_s^*$ is a smooth (possibly incomplete) hypersurface in $M$. By Proposition \ref{prop:crit.harm}, $\dim_\textrm{Haus}(\Sigma_s\setminus \Sigma_s^*) \leq n-2$. In particular, we see that $\cH^{n-1}(\Sigma_s^*) < \infty$ and for any function $f \in C^\infty_\textrm{loc}(M)$, 
\[
\int f d\cH^{n-1}\lfloor{\Sigma_s} = \int_{\Sigma_s^*} f 
\]
where the right hand side is taken with respect to the induced Riemannian volume form on $\Sigma_s^*$. 

\begin{lemm}\label{lemm:int.level.set.cont}
For $f \in C^0_\textnormal{loc}(M\setminus \{p\})$, the function $s \mapsto \int_{\Sigma_s^*} f$ is continuous. 
\end{lemm}
\begin{proof}
Cover $\Sigma_s\setminus\Sigma_s^*$ by balls $B_{r_i}(x_i)$ with $\sum_i r_i^{n-2+\eta} < \eps$ and $r_i \leq 1$. It is clear that 
\[
\delta \mapsto \int_{\Sigma_{s+\delta} \setminus \cup_i B_{r_i}(x_i)} f
\]
is continuous at $\delta = 0$. Furthermore, by Proposition \ref{prop:crit.harm}, we have 
\[
\int_{\Sigma_{s+\delta} \cap \cup_i B_{r_i}(x_i)} f \leq C \sum_i \cH^{n-1}(B_{r_i}(x_i)) \leq C \sum_i r_i^{n-1} \leq C \eps,
\]
where the constant $C$ is independent of $\delta$ small. Putting these facts together, the assertion follows. 
\end{proof}
We now consider the continuous functions
\begin{align}
F(s) := \int_{\Sigma_s^*} |\nabla u|^2 \label{eq:defi.F}\\
\cA(s) : = \int_{\Sigma_s^*} |A_M|^2 \label{eq:defi.A},
\end{align}
defined for $s \in (0,\infty)$. 

\begin{lemm}\label{lemm:F.abs.cont}
The function $F(s)$ is locally Lipschitz on $(0,\infty)$.
\end{lemm}
\begin{proof}
For a fixed compact subset $K\subset (0,\infty)$, consider $s<t$ regular values of $u$ with $s,t\in K$. Below we will allow the Lipschitz constant to depend on $K$.

Consider the region $\Omega_{s,t} = u^{-1}([s,t])$. Observe that $\Omega_{s,t}$ is a compact region with smooth boundary $\Sigma_s\cup\Sigma_t$. Write $\eta$ for the  outwards pointing unit normal to $\Omega_{s,t}$ and note that $\eta = \tfrac{\nabla u}{|\nabla u|}$ along $\Sigma_s$ and $\eta = -\tfrac{\nabla u}{|\nabla u|}$. We have
\begin{align*}
F(s)-F(t) & = \int_{\partial\Omega_{s,t}} \bangle{|\nabla u|\nabla u,\eta} \\
& = \lim_{\delta\to 0}\int_{\partial\Omega_{s,t}}  \bangle{(|\nabla u|^2 + \delta)^{\frac 12} \nabla u,\eta}\\
& = \lim_{\delta\to 0}\int_{\Omega_{s,t}}  \Div\left( (|\nabla u|^2 + \delta)^{\frac 12} \nabla u \right)\\
& = \lim_{\delta\to 0}\int_{\Omega_{s,t}}  -\tfrac 12  (|\nabla u|^2 + \delta)^{-\frac 12} \bangle{\nabla |\nabla  u|^2,\nabla u}.
\end{align*}
Note that 
\[
(|\nabla u|^2 + \delta)^{-\frac 12} \bangle{\nabla |\nabla  u|^2,\nabla u} \leq |\nabla |\nabla  u|^2| \leq 2 |\nabla u| |D^2u|,
\]
using the Kato inequality. From this, we find 
\[
|F(s) - F(t)| \leq \int_{\Omega_{s,t}} |\nabla u| |D^2 u| \leq C \int_s^t \cH^{n-1}(\Sigma_\tau) d\tau,
\]
using the co-area formula (where $C=C(M,g,u,K)$). Using Proposition \ref{prop:crit.harm}, to bound the volume of $\Sigma_\tau$, the assertion follows (the Lipchitz estimate for $s,t$ possibly singular follows from the above estimate and continuity of $F(t)$ proven in Lemma \ref{lemm:int.level.set.cont}). 
\end{proof}

\section{Stern's Bochner formula and stability} \label{sec:Stern}
Recently, Stern has discovered \cite{Stern} that one can combine the Bochner formula with the Schoen--Yau rearrangement \cite{SY:3d-torus} of the Gauss equation to relate the scalar curvature of a three-manifold to the behavior of a harmonic function on that manifold. In this section we consider this idea in the context of the stability operator for a stable minimal $3$-dimensional hypersurface (cf.\ \cite{SY:harmonic.stable.minimal}). 
 
We consider $M^3\to\RR^{4}$  a complete, connected, simply connected, two-sided, stable minimal immersion with uniformly bounded curvature. For $p\in M$ fixed, consider the Greens' function $u \in C^\infty_\textnormal{loc}(M\setminus\{p\})$ constructed in Proposition \ref{prop:Greens}. Recall the definition of $F(s),\cA(s)$ in \eqref{eq:defi.F} and \eqref{eq:defi.A}. By Lemma \ref{lemm:int.level.set.cont}, $F(s),\cA(s)$ are continuous on $(0,\infty)$. 

\begin{prop}\label{prop:stern.bochner.stable}
For any function $\varphi \in C^{0,1}_c((0,\infty))$ it holds 
\begin{align*}
& \int_{0}^{\infty} \varphi(s)^{2}  \cA(s)\, ds  \leq \frac{8\pi}{3} \int_{0}^{\infty} \varphi(s)^{2} \, ds  +  \frac 43 \int_{0}^{\infty} \varphi'(s)^{2} F(s) \, ds . 
\end{align*}
\end{prop}
\begin{proof}
Consider the following regularization of the square root of the gradient of $u$
\[
e_\delta = (|\nabla u|^2 + \delta)^{\frac 14}. 
\]
Note that $e_{\delta}\in C^{\infty}_{\textrm{loc}}(M\setminus\{p\})$ and
\[
\nabla e_{\delta} = \tfrac 1 4 e_{\delta}^{-3} \nabla |\nabla u|^{2},
\]
as well as
\[
\Delta e_{\delta} = \tfrac 1 4 e_{\delta}^{-3} \Delta |\nabla u|^{2} - \tfrac{3}{16} e_{\delta}^{-7} |\nabla |\nabla u|^{2}|^{2}.
\]
Applying the Bochner formula, we find 
\begin{align*}
\Delta e_{\delta} & = \tfrac 1 2 e_{\delta}^{-3} (|D^{2}u|^{2} + \Ric_{M}(\nabla u,\nabla u))- \tfrac{3}{16} e_{\delta}^{-7} |\nabla |\nabla u|^{2}|^{2}\\
& = \tfrac 1 2 e_{\delta}^{-3} (|D^{2}u|^{2} - \tfrac{3}{8} e_{\delta}^{-4} |\nabla |\nabla u|^{2}|^{2} + \Ric_{M}(\nabla u,\nabla u)). 
\end{align*}
We now consider $\psi \in C^\infty_0(M\setminus\{p\})$ and then take $f=e_\delta  \psi$ in the stability inequality \eqref{eq:stable} yielding
\begin{align*}
& \int_M |A_M|^2 e_\delta^2 \psi^2  \leq \int_M |\nabla e_\delta|^2 f^2 + \tfrac 12 \bangle{\nabla e_\delta^2 ,\nabla \psi^2} + e_\delta^2 |\nabla \psi|^2\\
& = \int_M |\nabla e_\delta|^2 f^2 - \tfrac 12 \psi^2 \Delta e_\delta^2  + e_\delta^2 |\nabla \psi|^2\\ 
& = \int_M  - e_\delta \psi^2 \Delta e_\delta  + e_\delta^2 |\nabla \psi|^2\\
& = \int_M - \tfrac 1 2 e_{\delta}^{-2} (|D^{2}u|^{2} - \tfrac{3}{8} e_{\delta}^{-4} |\nabla |\nabla u|^{2}|^{2} + \Ric_{M}(\nabla u,\nabla u))\psi^{2}   +  e_\delta^2 |\nabla \psi|^2.
\end{align*}
Rearranging this yields
\begin{align*}
& \int_{M} (|A_M|^2 e_\delta^2 +  \tfrac 1 2 e_{\delta}^{-2} (|D^{2}u|^{2} - \tfrac{3}{8} e_{\delta}^{-4} |\nabla |\nabla u|^{2}|^{2} + \Ric_{M}(\nabla u,\nabla u))) \psi^2  \leq \int_{M} e_\delta^2 |\nabla \psi|^2.
\end{align*}
Note that $\Ric_{M}\geq -|A_{M}|^{2}$ by the Gauss equations. Furthermore, the improved Kato inequality yields
\[
\tfrac 38 |\nabla |\nabla u|^{2}|^{2} \leq |\nabla u|^{2}|D^{2}u|^{2} \leq e_{\delta}^{4} |D^{2}u|^{2}.
\]
From this, we find that 
\[
|A_M|^2 e_\delta^2 +  \tfrac 1 2 e_{\delta}^{-2} (|D^{2}u|^{2} - \tfrac{3}{8} e_{\delta}^{-4} |\nabla |\nabla u|^{2}|^{2} + \Ric_{M}(\nabla u,\nabla u)) \geq 0.
\]
Let $\fB$ be an open subset of $(0,\infty)$ containing all singular values of $u$, and $\fA=(0,\infty)\setminus \fB$. We thus find
\begin{align*}
& \int_{\fA} (|A_M|^2 e_\delta^2 +  \tfrac 1 2 e_{\delta}^{-2} (|D^{2}u|^{2} - \tfrac{3}{8} e_{\delta}^{-4} |\nabla |\nabla u|^{2}|^{2} + \Ric_{M}(\nabla u,\nabla u))) \psi^2\\
& \leq \int_{\fA} e_\delta^2 |\nabla \psi|^2 + \int_{\fB} e_\delta^2 |\nabla \psi|^2
\end{align*}
We can send $\delta \to 0$ and apply Fatou's lemma to find 
\begin{align*}
& \int_{u^{-1}(\fA)} (|A_M|^2 |\nabla u| +  \tfrac 1 2 |\nabla u|^{-1} (|D^{2}u|^{2} - \tfrac{3}{2}  |\nabla |\nabla u||^{2} + \Ric_{M}(\nabla u,\nabla u))) \psi^2\\
& \leq \int_{u^{-1}(\fA)} |\nabla u| |\nabla \psi|^2 + \int_{u^{-1}(\fB)} |\nabla u| |\nabla \psi|^2
\end{align*}
By the co-area formula, we find 
\begin{align*}
& \int_{\fA} \left( \int_{\Sigma_{s}} (|A_M|^2  +  \tfrac 1 2 |\nabla u|^{-2} (|D^{2}u|^{2} - \tfrac{3}{2}  |\nabla |\nabla u||^{2} + \Ric_{M}(\nabla u,\nabla u))) \psi^2\right) ds\\
& \leq \int_{\fA} \left( \int_{\Sigma_{s}} |\nabla \psi|^2 \right) ds + \int_{u^{-1}(\fB)} |\nabla u| |\nabla \psi|^2.
\end{align*}
For $\varphi \in C_{c}^{\infty}((0,\infty))$, take $\psi = \varphi(u)$ to yield
\begin{align*}
& \int_{\fA} \varphi(s)^{2} \left( \int_{\Sigma_{s}} |A_M|^2  +  \tfrac 1 2 |\nabla u|^{-2} (|D^{2}u|^{2} - \tfrac{3}{2}  |\nabla |\nabla u||^{2} + \Ric_{M}(\nabla u,\nabla u)) \right) ds\\
& \leq \int_{\fA} \varphi'(s)^{2} \left( \int_{\Sigma_{s}} |\nabla u|^2\right) ds+ \int_{u^{-1}(\fB)} \varphi'(u)^{2} |\nabla u|^{3} 
\end{align*}

If $s \in \fA$, note that $\nu = \frac{\nabla u}{|\nabla u|}$ is a unit normal to $\Sigma_{s}$. We now follow the ideas used in \cite[Theorem 1]{Stern}. The traced Gauss equations yield
\[
2 \Ric_M(\nu,\nu) = R_M - 2K_{\Sigma_s} - |A_{\Sigma_s}|^2 + H_{\Sigma_s}^2. 
\]
Similarly, using the Gauss equations for $M^3\to \RR^4$ we have $R_M = - |A_M|^2$. The scalar second fundamental form of $\Sigma_s$ satisfies
\[
A_{\Sigma_s} = |\nabla u|^{-1} D^2 u|_{\Sigma_s}
\]
so
\[
|\nabla u|^2 |A_{\Sigma_s}|^2 = |D^2u|^2 - 2 |\nabla |\nabla u||^2 + D^2 u(\nu,\nu)^2
\]
and (because $u$ is harmonic) 
\[
|\nabla u|^2 H_{\Sigma_s}^2 = D^2 u(\nu,\nu)^2.
\]
Thus, 
\[
\Ric_M(\nabla u,\nabla u) = - \tfrac 12 |\nabla u|^{2} |A_{M}|^{2} - |\nabla u|^{2} K_{\Sigma_s} - \tfrac 12 |D^2u|^2 +  |\nabla |\nabla u||^2. 
\]
along $\Sigma_{s}$. Thus,
\begin{align*}
& \int_{\fA} \varphi(s)^{2} \left( \int_{\Sigma_{s}} \tfrac 34 |A_M|^2  +    \tfrac 1 4 |\nabla u|^{-2}( |D^{2}u|^{2} -   |\nabla |\nabla u||^{2})  \right) ds  \\
& \leq \int_{\fA} \varphi(s)^{2} \left( \int_{\Sigma_{s}} \tfrac 12 K_{\Sigma_{s}} \right) ds+   \int_{\fA} \varphi'(s)^{2} \left( \int_{\Sigma_{s}} |\nabla u|^2 \right) ds  + \int_{u^{-1}(\fB)} \varphi'(u)^{2} |\nabla u|^{3} 
\end{align*}
By (6) in Proposition \ref{prop:Greens}, for $s\in \cR$, $\Sigma_{s}$ is connected, so $
\int_{\Sigma_{s}} K_{\Sigma_{s}} \leq 4\pi$ (by Gauss--Bonnet). Using this, the Kato inequality, and the definition of $F(s),\cA(s)$, we find
\begin{align*}
& \int_{\fA}   \varphi(s)^{2}  \cA(s)\, ds \leq \frac{8\pi}{3} \int_{\fA} \varphi(s)^{2} \, ds  +  \frac 43 \int_{\fA} \varphi'(s)^{2} F(s) \, ds +\frac 43  \int_{u^{-1}(\fB)} \varphi'(u)^{2} |\nabla u|^{3} 
\end{align*}
Since $\varphi'(u)^{2} |\nabla u|^{3}$ is uniformly bounded, we can send $|\fB \cap \supp \varphi| \to 0$ and conclude 
\begin{align*}
& \int_{0}^{\infty} \varphi(s)^{2}  \cA(s) \, ds \leq \frac{8\pi}{3} \int_{0}^{\infty} \varphi(s)^{2}   \, ds +  \frac 43 \int_{0}^{\infty} \varphi'(s)^{2} F(s)  \, ds. 
\end{align*}
A standard approximation argument for $\varphi$ completes the proof. 
\end{proof}

\begin{rema}
By using the improved Kato inequality it is easy to generalize the previous argument to prove that
\[
\int_0^\infty \varphi(s)^2 \left( \int_{\Sigma_s} \tfrac 34 |A_M|^2 + \tfrac{1}{12} |A_{\Sigma_s} |^2 -K_{\Sigma_s} \right) ds \leq \int_0^\infty \varphi'(s)^2 F(s) \, ds,
\]
where $A_{\Sigma_s}$ is the second fundamental form of $\Sigma_s$ in $M$. However, we will not need this expression in the sequel. 
\end{rema}

We thus see that the behavior of $F(s)$ near $0$ and $\infty$ determines the potential test functions $\varphi$ that can be used in the inequality from Proposition \ref{prop:stern.bochner.stable}. The behavior of $F(t)$ as $t\to\infty$ is independent of the geometry of $M$ ($F(t)$ behaves like the Green's function on $\RR^3$). However, the behavior as $t\to0$ is of crucial importance to our argument. We begin with an \emph{a priori} bound that will be improved in the sequel. 

\begin{lemm}\label{lemm:F.apriori.bds}
It holds that $F(t) = O(t)$ as $t\to 0$ and $F(t) = (1+o(1))4\pi t^2$ as $t\to\infty$. 
\end{lemm}
\begin{proof}
The bound as $t\to\infty$ follows in a straightforward manner from the asymptotics in (3) in Proposition \ref{prop:Greens}. 

For the other assertion, by integrating $\Delta u =0$ over $\{t\leq u \leq \tau\}$ (where $t,\tau\in \cR$) we find
\[
\int_{\Sigma_t} |\nabla u| = \int_{\Sigma_\tau} |\nabla u|. 
\]
By Lemma \ref{lemm:int.level.set.cont}, $t\mapsto \int_{\Sigma_t} |\nabla u| $ is constant. Because $M$ has uniformly bounded Ricci curvature, we can apply the Harnack inequality (Proposition \ref{prop:diff.harnack}) to obtain
\[
F(t) = \int_{\Sigma_t} |\nabla u|^2 \leq C t \int_{\Sigma_t} |\nabla u|. 
\]
(for $t$ bounded away from $\infty$). This proves the assertion. 
\end{proof}

\begin{coro}\label{coro:cutoff}
We have
\begin{align*}
& \limsup_{\ell\searrow 0} \int_{\ell}^{t}   \cA(s)\, ds + t^2 \int_{t}^{1}  s^{-2}  \cA(s)\, ds   \leq O(t)  +  \frac 43 \int_{t}^{1} t^2 s^{-4} F(s) \, ds
\end{align*}
as $t\to 0$. 
\end{coro}
\begin{proof}
For $\eps \in (0,1)$ and $\ell < t$, we consider the following function
\[
\varphi_\eps(s) = 
\begin{cases}
0 & s \in (0,\eps\ell) \\
1-\tfrac{\log s - \log\ell}{\log \eps} & s \in [\eps\ell,\ell)\\
1 & s \in [\ell,t)\\
ts^{-1}  & s \in [t,1)\\
t(2-s) & s \in [1,2)\\
0 & s \in (2,\infty). 
\end{cases}
\]
We have that
\begin{align*}
& \int_0^1 \varphi_\eps(s)^2 \, ds = O(t).
\end{align*} 
Using $F(t) = O(t)$ as $t\to 0$ from Lemma \ref{lemm:F.apriori.bds}, we can further estimate 
\[
\int_0^{\ell} \varphi_\eps'(s)^2 F(s) \, ds = O(|\log\eps|^{-2})\int_{\eps\ell}^\ell s^{-1} \, ds= O(|\log \eps|^{-1}).
\]
As such, taking $\varphi_\eps$ in Proposition \ref{prop:stern.bochner.stable} and sending $\eps\to 0$ we find 
\begin{align*}
& \int_{\ell}^{t}   \cA(s)\, ds + t^2 \int_{t}^{1}  s^{-2}  \cA(s)\, ds   \leq O(t)  +  \frac 43 \int_{t}^{1} t^2 s^{-4} F(s) \, ds,
\end{align*}
where the $O(t)$ is bounded independently of $\ell>0$. This completes the proof. 
\end{proof}

\section{An extension of Munteanu--Wang's monotonicity formula}\label{sec:MW}

Munteanu--Wang have recently established \cite{MunteanuWang} a sharp monotonicity formula for the quantity $F(s)$ defined via a Green's function on a non-parabolic three-manifold with non-negative scalar curvature. In this section we refine this estimate in two ways to apply in the present situation. First of all, the minimal hypersurface does not have non-negative scalar curvature so we must keep track of the resulting defect term. Second, we have to regularize their argument to account for our relatively weak \emph{a priori} bounds for $F(s)$ as $s\to 0$.

We continue to consider $M^3\to\RR^{4}$  a complete, connected, simply connected,  two-sided stable minimal immersion with uniformly bounded curvature. For $p\in M$ fixed, consider the Greens' function $u \in C^\infty_\textnormal{loc}(M\setminus\{p\})$ constructed in Proposition \ref{prop:Greens}. Recall the definition of $F(s),\cA(s)$ in \eqref{eq:defi.F} and \eqref{eq:defi.A}. By Lemma \ref{lemm:int.level.set.cont}, $F(s),\cA(s)$ are continuous on $(0,\infty)$. Furthermore, by Lemma \ref{lemm:F.abs.cont}, $F(s)$ is locally absolutely continuous. Finally, if $\cR$ is the open and dense (by Lemma \ref{lemm:reg.open.dense}) set of regular values of $u$, we see that $F(s)$ is differentiable at all $s \in \cR$.

We consider $\alpha \mapsto \lambda(\alpha)$ where
\begin{equation}\label{eq:lambda}
\lambda(\alpha) : = \alpha + 1 - \sqrt{(\alpha+1)(1-\tfrac 13 \alpha)}.
\end{equation}
Note that $\lambda(\alpha) \in \RR$ for $\alpha \in [-1,3]$. 

\begin{lemm}\label{lemm:alpha}
There is $\alpha_0 \in (1,2)$ so that if $\alpha \in (\alpha_0,2)$ then $\alpha - \tfrac 32 \lambda(\alpha) + 1 > 0$. 
\end{lemm} 
\begin{proof}
Taylor expanding around $\alpha =2$ yields
\[
\alpha - \tfrac 32 \lambda(\alpha) + 1 = - (\alpha-2) + O((\alpha-2)^2). 
\]
This proves the assertion. 
\end{proof}

\begin{prop}[{cf.\ \cite[Theorem 3.1]{MunteanuWang}}] \label{prop:MW-monotonicity}
There is $C>0$ so that for $t \in (0,1)$, it holds that 
\begin{align*}
F(t)  
& \leq  C t^{3}   +  4\pi t^2  + \frac 14 t \liminf_{\ell\searrow 0} \int_\ell^t   \cA(s)\, ds  + \frac 14 t^{3}   \int_t^1 s^{-2} \cA(s)\, ds.
\end{align*}
\end{prop}
\begin{proof}
We begin by considering $t\in \cR$. Set $\nu = \frac{\nabla u}{|\nabla u|}$ (note that $\nu$ is \emph{inwards pointing} for the set $\Omega_s = \{u\geq s\}$). The family $t\mapsto \Sigma_t$ has normal velocity $|\nabla u|^{-1}\nu$. Furthermore, the mean curvature of $\Sigma_s$ satisfies
\[
H = - |\nabla u|^{-1}\bangle{\nabla |\nabla u|,\nu}. 
\]
With this convention, we have
\begin{align}
F'(t) & = \int_{\Sigma_t} |\nabla u|^{-1} \bangle{\nabla |\nabla u|^2,\nu} + |\nabla u|^{-1} H |\nabla u|^2  = \int_{\Sigma_t}  \bangle{\nabla |\nabla u|,\nu} \label{eq:F.prime}
\end{align}
Fix $\alpha \in (\alpha_0,2)$ as in Lemma \ref{lemm:alpha} and write $\lambda = \lambda(\alpha)$ (defined in \eqref{eq:lambda}). Note that 
\begin{align*}
t^{-\alpha} F'(t) + \alpha t^{-\alpha-1}F(t) & = \int_{\Sigma_t} u^{-\alpha} \bangle{\nabla |\nabla u|,\nu} + \alpha u^{-\alpha-1} |\nabla u|^2\\
& = \int_{\Sigma_t} u^{-\alpha} \bangle{\nabla |\nabla u|,\nu} - |\nabla u| \bangle{\nabla u^{-\alpha},\nu}.
\end{align*}
Set
\[
j_\delta = (|\nabla u|^2 + \delta)^{\frac 12} 
\]
so that
\begin{align*}
t^{-\alpha} F'(t) + \alpha t^{-\alpha-1}F(t) & = \lim_{\delta\searrow 0}  \int_{\Sigma_t} u^{-\alpha} \bangle{\nabla j_\delta,\nu} - j_\delta \bangle{\nabla u^{-\alpha},\nu}.
\end{align*}

Now, if $0<t<\tau\leq 1$, $t,\tau \in \cR$, Green's second identity on $\Omega_{t,\tau} : = \{t \leq u \leq \tau\}$ thus yields
\begin{align*}
& \int_{\Sigma_t} u^{-\alpha} \bangle{\nabla j_\delta,\nu} - j_\delta \bangle{\nabla u^{-\alpha},\nu} - \int_{\Sigma_\tau} u^{-\alpha} \bangle{\nabla j_\delta,\nu} - j_\delta \bangle{\nabla u^{-\alpha},\nu} \\
& = - \int_{\Omega_{t,\tau}} u^{-\alpha} \Delta j_\delta - j_\delta \Delta u^{-\alpha} .
\end{align*}
(Recall that $\nu$ is inwards pointing for $\Omega_s$.) Because $u$ is harmonic 
\[
\Delta u^{-\alpha} = \alpha(\alpha+1) u^{-\alpha-2}|\nabla u|^2.
\] 
Furthermore, the Bochner formula yields 
\[
 \Delta j_\delta = j_\delta^{-1} (|D^2 u|^2 - \tfrac 1 4 j_\delta^{-2} |\nabla |\nabla u|^2|^2) + j_\delta^{-1} \Ric_M(\nabla u,\nabla u). 
\]
Thus, we find 
\begin{align*}
& \int_{\Sigma_t} u^{-\alpha} \bangle{\nabla j_\delta,\nu} - j_\delta \bangle{\nabla u^{-\alpha},\nu} - \int_{\Sigma_\tau} u^{-\alpha} \bangle{\nabla j_\delta,\nu} - j_\delta \bangle{\nabla u^{-\alpha},\nu} \\
& = - \int_{\Omega_{t,\tau}} u^{-\alpha} j_\delta^{-1} (|D^2 u|^2 - \tfrac 1 4 j_\delta^{-2} |\nabla |\nabla u|^2|^2) + u^{-\alpha} j_\delta^{-1} \Ric_M(\nabla u,\nabla u)  \\
& \qquad +  \int_{\Omega_{t,\tau}}  \alpha(\alpha+1) u^{-\alpha-2} j_\delta |\nabla u|^2.
\end{align*}
Note that the Kato inequality yields
\[
|D^2 u|^2 - \tfrac 1 4 j_\delta^{-2} |\nabla |\nabla u|^2|^2 \geq 0. 
\]
Let $\fB$ be an open subset of $(0,\infty)$ containing all singular values of $u$, and $\fA=(0,\infty)\setminus \fB$. We thus find
\begin{align*}
& \int_{\Sigma_t} u^{-\alpha} \bangle{\nabla j_\delta,\nu} - j_\delta \bangle{\nabla u^{-\alpha},\nu} - \int_{\Sigma_\tau} u^{-\alpha} \bangle{\nabla j_\delta,\nu} - j_\delta \bangle{\nabla u^{-\alpha},\nu} \\
& \leq - \int_{(t,\tau)\cap \fA} \left( \int_{\Sigma_s} u^{-\alpha} |\nabla u|^{-1} j_\delta^{-1} (|D^2 u|^2 - \tfrac 1 4 j_\delta^{-2} |\nabla |\nabla u|^2|^2)  \right) ds \\
& \qquad -   \int_{(t,\tau)\cap \fA} \left( \int_{\Sigma_s}  u^{-\alpha} |\nabla u|^{-1} j_\delta^{-1} \Ric_M(\nabla u,\nabla u)   \right) ds \\
& \qquad +   \int_{(t,\tau)\cap \fA} \left( \int_{\Sigma_s}   \alpha(\alpha+1) u^{-\alpha-2} j_\delta |\nabla u| \right) ds \\
& \qquad + \int_{\Omega_{t,\tau} \cap u^{-1}(\fB)} \alpha(\alpha+1) u^{-\alpha-2} j_\delta |\nabla u|^2 -  u^{-\alpha} j_\delta^{-1} \Ric_M(\nabla u,\nabla u)  
\end{align*}
Because $\alpha(\alpha+1) u^{-\alpha-2} j_\delta |\nabla u|^2 -  u^{-\alpha} j_\delta^{-1} \Ric_M(\nabla u,\nabla u)$ is uniformly bounded in $L^\infty(\Omega_{t,\tau})$ as $\delta\to 0$, we can send $\delta\to0$ and then $|\cB| \to 0$ to find 
\begin{align*}
& (t^{-\alpha} F'(t) + \alpha t^{-\alpha-1}F(t)) - (\tau^{-\alpha} F'(\tau) + \alpha \tau^{-\alpha-1}F(\tau)) \\
& \leq - \int_t^\tau \left( \int_{\Sigma_s} u^{-\alpha} |\nabla u|^{-2}   (|D^2 u|^2 -  |\nabla |\nabla u||^2)  \right) ds \\
& \qquad -   \int_t^\tau \left( \int_{\Sigma_s}  u^{-\alpha} |\nabla u|^{-2}  \Ric_M(\nabla u,\nabla u)   \right) ds \\
& \qquad +   \int_t^\tau \left( \int_{\Sigma_s}   \alpha(\alpha+1) u^{-\alpha-2}   |\nabla u|^2 \right) ds . 
\end{align*}
Using Stern's rearrangement \cite{Stern} of the Bochner terms as in the proof of Proposition \ref{prop:stern.bochner.stable} we can write (along $\Sigma_s$, $s \in \cR$)
\[
|\nabla u|^{-2} \Ric_M(\nabla u,\nabla u) = -\tfrac 12 |A_M|^2 - K_{\Sigma_s} + |\nabla u|^{-2} |\nabla |\nabla u||^2 - \tfrac 12|\nabla u|^{-2} |D^2 u|^2
\]
so (also using that $u$ is constant along its level sets)
\begin{align*}
& \left(t^{-\alpha} F'(t) + \alpha t^{-\alpha-1}F(t)\right) -\left (\tau^{-\alpha} F'(\tau) + \alpha \tau^{-\alpha-1}F(\tau)\right) \\
& \leq \frac 12  \int_t^\tau s^{-\alpha} \left( \int_{\Sigma_s}   |A_M|^2\right)+ \int_t^\tau s^{-\alpha} \left( \int_{\Sigma_s}    K_{\Sigma_s}   \right) ds \\
& \qquad - \frac 12 \int_t^\tau s^{-\alpha} \left( \int_{\Sigma_s}  |\nabla u|^{-2}   |D^2 u|^2   \right) ds  +   \int_t^\tau \alpha(\alpha+1) s^{-\alpha-2} F(s) ds.
\end{align*}
By Gauss--Bonnet and Proposition \ref{prop:Greens},
\[
\int_t^\tau s^{-\alpha} \left( \int_{\Sigma_s}    K_{\Sigma_s}   \right) ds \leq \frac{1}{\alpha-1} 4\pi (t^{1-\alpha} - \tau^{1-\alpha}).
\]
Furthermore, using the improved Kato inequalty, as well as the Cauchy--Schwarz and H\"older inequalities on \eqref{eq:F.prime}, we find 
\begin{align*}
\int_{\Sigma_s}  |\nabla u|^{-2}   |D^2 u|^2 & \geq  \int_{\Sigma_s} \tfrac 32 |\nabla u|^{-2} |\nabla |\nabla u||^2 \\
& \geq  \int_{\Sigma_s} \tfrac 32 |\nabla u|^{-2} \bangle{\nabla |\nabla u|,\nu}^2\\
&  \geq \tfrac 32 F(s)^{-1} F'(s)^2
\end{align*}
for $s \in \cR$. Putting this together we find 
\begin{align*}
& \left(t^{-\alpha} F'(t) + \alpha t^{-\alpha-1}F(t)\right) - \left(\tau^{-\alpha} F'(\tau) + \alpha \tau^{-\alpha-1}F(\tau)\right) \\
& \leq \frac 12  \int_t^\tau s^{-\alpha} \cA(s) \, ds+ \frac{1}{\alpha-1} 4\pi (t^{1-\alpha} - \tau^{1-\alpha})\\
& \qquad +  \int_t^\tau  ( -\tfrac 34 s^{-\alpha}  F(s)^{-1} F'(s)^2 + \alpha(\alpha+1) s^{-\alpha-2} F(s))ds
\end{align*}
Cauchy--Schwarz yields
\[
2s^{-1} F'(s) F(s) \leq \lambda^{-1} F'(s)^2 + \lambda s^{-2} F(s)^2,
\]
(where $\lambda=\lambda(\alpha)$ as defined in \eqref{eq:lambda}). Rearranging, we find
\[
-\tfrac 34s^{-\alpha} F(s)^{-1} F'(s)^2\leq - \tfrac 3 2  \lambda \tau^{-\alpha-1} F'(s) + \tfrac 34 \lambda^2 s^{-\alpha-2} F(s). 
\]
Using this above we find 
\begin{align*}
& \left(t^{-\alpha} F'(t) + \alpha t^{-\alpha-1}F(t)\right) - \left(\tau^{-\alpha} F'(\tau) + \alpha \tau^{-\alpha-1}F(\tau)\right) \\
& \leq \frac 12  \int_t^\tau s^{-\alpha} \cA(s)\, ds + \frac{1}{\alpha-1} 4\pi (t^{1-\alpha} - \tau^{1-\alpha}) \\
& \qquad +  \int_t^\tau  (- \tfrac 3 2  \lambda s^{-\alpha-1} F'(s) + (\tfrac 34 \lambda^2 + \alpha(\alpha+1) ) s^{-\alpha-2} F(s)) \, ds \\
& = \frac 12  \int_t^\tau s^{-\alpha} \cA(s)\, ds + \frac{1}{\alpha-1} 4\pi (t^{1-\alpha} - \tau^{1-\alpha}) \\
& \qquad +  \int_t^\tau    (\tfrac 34 \lambda^2 - \tfrac 3 2\lambda  (\alpha+1)  + \alpha(\alpha+1) ) s^{-\alpha-2} F(s) \, ds \\
& \qquad - \tfrac 32 \lambda \tau^{-\alpha-1}F(\tau) + \tfrac 32 \lambda t^{-\alpha-1} F(t),
\end{align*}
where we integrated by parts in the last step. Observe that 
\[
\tfrac 34 \lambda^2 - \tfrac 3 2\lambda  (\alpha+1)  + \alpha(\alpha+1) = 0
\]
by \eqref{eq:lambda}, so we can rearrange this to read
\begin{align*}
& \left(t^{-\alpha} F'(t) + (\alpha-\tfrac 32 \lambda) t^{-\alpha-1}F(t)\right) - \left(\tau^{-\alpha} F'(\tau) + (\alpha-\tfrac 32 \lambda)  \tau^{-\alpha-1}F(\tau)\right) \\
& \leq \frac 12  \int_t^\tau s^{-\alpha} \cA(s)\, ds + \frac{1}{\alpha-1} 4\pi (t^{1-\alpha} - \tau^{1-\alpha})\\
& \leq \frac 12  \int_t^\tau s^{-\alpha} \cA(s)\, ds + \frac{1}{\alpha-1} 4\pi t^{1-\alpha}  .
\end{align*}
We thus find 
\begin{align*}
\left(t^{\alpha-\tfrac 32 \lambda} F(t)\right)'
& \leq C t^{2\alpha-\tfrac 32 \lambda} + \frac{1}{\alpha-1} 4\pi  t^{\alpha-\tfrac 32 \lambda + 1}  +  \frac 12 t^{2\alpha-\tfrac 32 \lambda} \int_t^\tau s^{-\alpha} \cA(s)\, ds \end{align*}
where $C=C(\tau)$ is bounded uniformly for $\alpha \in (\alpha_0,2)$. By Lemma \ref{lemm:F.abs.cont} we can integrate this on $(\ell,t)$ for $\ell,t \in \cR$, $\ell<t <\tau$, yielding 
\begin{align*}
t^{\alpha-\tfrac 32 \lambda} F(t)  
& \leq \ell^{\alpha-\tfrac 32 \lambda} F(\ell) + \frac{C}{2\alpha- \tfrac 32\lambda+1} t^{2\alpha-\tfrac 32 \lambda+1} \\
& \qquad  + \frac{1}{(\alpha-1)(\alpha-\tfrac 32 \lambda + 2)} 4\pi t^{\alpha-\tfrac 32 \lambda + 2}  \\
& \qquad +  \frac 12 \int_\ell^t  \int_\sigma^\tau \sigma^{2\alpha-\tfrac 32 \lambda} s^{-\alpha} \cA(s)\, ds d\sigma 
\end{align*}
(where we have dropped several negative terms evaluated at $\ell$, cf.\ Lemma \ref{lemm:alpha}). We apply Fubini's theorem to write 
\begin{align*}
& \int_\ell^t  \int_\sigma^\tau \sigma^{2\alpha-\tfrac 32 \lambda} s^{-\alpha} \cA(s)\, ds d\sigma \\
& = \int_\ell^t \int_\ell^s  \sigma^{2\alpha-\tfrac 32 \lambda} s^{-\alpha} \cA(s)\, d\sigma ds + \int_t^\tau \int_\ell^t  \sigma^{2\alpha-\tfrac 32 \lambda} s^{-\alpha} \cA(s)\, d\sigma ds \\
& \leq \frac{1}{2\alpha-\frac 32\lambda+1} \int_\ell^t  s^{\alpha-\tfrac 32 \lambda+1} \cA(s)\, ds +\frac{1}{2\alpha-\frac 32\lambda+1} t^{2\alpha-\tfrac 32 \lambda + 1}  \int_t^\tau   s^{-\alpha} \cA(s)\, ds. 
\end{align*}
By Lemma \ref{lemm:alpha}, $s^{\alpha - \tfrac 32 \lambda + 1} \leq 1$  for $s \in (0,1)$. Because $t < 1$ we can thus estimate
\begin{align*}
& \int_\ell^t  \int_\sigma^\tau \sigma^{2\alpha-\tfrac 32 \lambda} s^{-\alpha} \cA(s)\, ds d\sigma \\
& \leq \frac{1}{2\alpha-\frac 32\lambda+1} \int_\ell^t    \cA(s)\, ds +\frac{1}{2\alpha-\frac 32\lambda+1} t^{2\alpha-\tfrac 32 \lambda + 1}  \int_t^\tau   s^{-\alpha} \cA(s)\, ds. 
\end{align*}
On the other hand, by Lemmas \ref{lemm:F.apriori.bds} and \ref{lemm:alpha} we have that 
\[
\ell^{\alpha-\tfrac 32 \lambda} F(\ell) = O(\ell^{\alpha-\tfrac 32 \lambda+1}) = o(1)
\]
as $\ell\to0$. Thus, we can pass to the limit as $\ell \searrow 0$ to obtain
\begin{align*}
t^{\alpha-\tfrac 32 \lambda} F(t)  
& \leq  \frac{C}{2\alpha- \tfrac 32\lambda+1} t^{2\alpha-\tfrac 32 \lambda+1}   + \frac{1}{(\alpha-1)(\alpha-\tfrac 32 \lambda + 2)} 4\pi t^{\alpha-\tfrac 32 \lambda + 2}  \\
& \qquad +  \frac{1}{2(2\alpha-\frac 32\lambda+1)} \liminf_{\ell\searrow 0} \int_\ell^t    \cA(s)\, ds \\
& \qquad +\frac{1}{2(2\alpha-\frac 32\lambda+1)} t^{2\alpha-\tfrac 32 \lambda + 1}   \int_t^\tau   s^{-\alpha} \cA(s)\, ds.
\end{align*}
Because $\alpha \mapsto \lambda(\alpha)$ is continuous at $\alpha=2$ and $\lambda(2) = 2$ we can then send $\alpha \nearrow 2$ to find 
\begin{align*}
t^{-1} F(t)  
& \leq  C t^{2}   +  4\pi t  + \frac 14  \liminf_{\ell\searrow 0} \int_\ell^t   \cA(s)\, ds  + \frac 14 t^{2}   \int_t^\tau   s^{-2} \cA(s)\, ds
\end{align*}
Because $\tau \leq 1$, this yields the assertion. 
\end{proof}
\begin{rema}\label{rema:MW.ric}
In the first version of this article, we suggested that the method of proof of Proposition \ref{prop:MW-monotonicity} should be capable of weakening the hypothesis $\liminf_{x\to\infty}\Ric \geq 0$ in \cite[Corollary 3.2]{MunteanuWang} to $\liminf_{x\to\infty}\Ric > - \infty$. This has recently been carried out in a more general context in \cite{CCLT}.
\end{rema}

Finally, we are able to obtain the following sharp decay estimate. 
\begin{coro}\label{coro:F.est.sharp}
We have $F(t) = O(t^2)$ as $t\to 0$.
\end{coro}
\begin{proof}
Combining Proposition \ref{prop:MW-monotonicity} with Corollary \ref{coro:cutoff} we obtain
\begin{equation}\label{eq:bootstrap.bd.F}
F(t)  \leq  O(t^2)  +  \frac 13 t^3 \int_{t}^{1}  s^{-4} F(s) \, ds
\end{equation}
as $t\to 0$. Assume for contradiction that $\limsup_{t\searrow 0} F(t) t^{-2} = \infty$. If this held, then we could choose $\{t_j\}_{j\in\NN} \subset (0,1)$ so that
\[
F(t_j) t_j^{-2} = \max_{[t_j,1]} F(s) s^{-2} \to \infty. 
\]
Using \eqref{eq:bootstrap.bd.F} at $t=t_j$ we find 
\begin{align*}
F(t_j) & \leq  O(t_j^2)  +  \frac 13 t_j^3 \int_{t_j}^{1}  s^{-2} (s^{-2} F(s)) \, ds\\
& \leq  O(t_j^2)  +  \frac 13 t_j  F(t_j) \int_{t_j}^{1}  s^{-2} \, ds\\
& =  O(t_j^2)  +  \frac 13 F(t_j)  (1-t_j). 
\end{align*}
Rearranging this we find
\[
(1-\tfrac 13 (1-o(1))) F(t_j) \leq O(t_j^2). 
\]
This is a contradiction, completing the proof. 
\end{proof} 

\section{Proof of Theorems \ref{theo:bernstein}, \ref{theo:curvature}, and \ref{theo:curvature.Riem}}\label{sec:proofs}

Theorem \ref{theo:bernstein} follows from Theorem \ref{theo:curvature} and the fact that there are no compact minimal surfaces in $\RR^{n+1}$. To prove Theorems \ref{theo:curvature} and \ref{theo:curvature.Riem}, we briefly recall a standard point-picking argument (cf.\ \cite[Lecture 3]{White:notes}) as follows: if Theorem \ref{theo:curvature} (or Theorem \ref{theo:curvature.Riem}) was not true, then we could find a sequence of two-sided, stable minimal immersed hypersurfaces $M_i$ in $\RR^4$ (or in $(N^4,g)$) and $p_i\in M_i$ such that
\[|A_{M_{i}}(p_i)| d_{M_{i}}(p_i,\partial M_i)=R_i\to \infty.\]
Here the distance is intrinsic on $M_i$. By considering an appropriate subset of $M_i$ we can assume that $M_i$ is compact and smooth up to its boundary. This allows us to assume that $p_i$ maximizes $|A_i(x)| d_M(x,\partial M_i)$. By translating and rescaling (we still denote the surfaces by $M_i$), we can ensure that $p_i=0$ and $|A_i(0)|=1$. Then, for any $r<R_i$ and $x \in M_{i}$ with $d_{M_{i}}(0,x)\le r$ we find
\[|A_{M_{i}}(x)|\le \frac{R_i}{d_{M_{i}}(x,\partial M_i)}\le \frac{R_i}{R_i-r},\]
and thus for each $r>0$, 
\[\sup_{d_{M_i}(x,0)\le r} |A_{M_{i}}(x)|\le \frac{R_i}{R_i-r}\to 1.\]
Therefore $M_i$ subsequentially converges smoothly to a complete, two-sided, stable minimal immersion $M^3 \to \RR^4$ with $|A_M(0)|=1$ and $|A_M(x)|\le 1$ for all $x\in M$ (this last condition was not assumed \emph{a priori} in the statement of Theorem \ref{theo:bernstein}).

To show that such an immersion does not exist, we first pass to the universal cover to arrange that assume that $M$ is simply connected (two-sided stability passes to any covering space by \cite[Theorem 1]{fischer-colbrie-schoen}). We can construct the Green's function $u \in C^{\infty}_{\textrm{loc}}(M\setminus\{p\})$ as in Proposition \ref{prop:Greens} and conclude that
\[
F(t) = \int_{\Sigma_{t}} |\nabla u|^{2}
\]
satisfies $F(t) \leq Ct^{2}$ for all $t \in (0,\infty)$ by Lemma \ref{lemm:F.apriori.bds} and Corollary \ref{coro:F.est.sharp}. Consider $f=\varphi \circ u$ for $\varphi \in C^{0,1}_{c}((0,\infty))$ in Proposition \ref{prop:SSY.L3} and apply the co-area formula to write
\begin{align*}
\int_{M} |A_{M}|^{3} \varphi(u)^{3} & \leq C  \int_{M} |\nabla (\varphi \circ u)|^{3}\\
& = C \int_{M} \varphi'(u)^{3} |\nabla u|^{3}\\
& = C \int_{0}^{\infty} \varphi'(s)^{3} \left( \int_{\Sigma_{s}} |\nabla u|^{2} \right) ds\\
& = C \int_{0}^{\infty} \varphi'(s)^{3} F(s) \, ds \\
& \leq C \int_{0}^{\infty} \varphi'(s)^{3} s^{2} \, ds.
\end{align*}
For $\rho \gg0$ choose
\[
\varphi(t) = \begin{cases}
0 & t \in [0,\rho^{-2})\\
2 + \tfrac{\log t}{\log \rho }  & t\in [\rho^{-2},\rho^{-1})\\
1 & t \in [\rho^{-1},\rho)\\
2-\tfrac{\log t}{\log \rho} & t\in [\rho,\rho^2)\\
0 & t \in [\rho^2,\infty). 
\end{cases}
\]
We find
\begin{align*}
& \int_{\{\rho^{-1}\leq u \leq \rho\}} |A_{M}|^3 \\
& \leq C \int_{\rho^{-2}}^{\rho^{-1}} \frac{s^2}{s^3 |\log \rho|^3} \, ds + C \int_{\rho}^{\rho^{2}} \frac{s^2}{s^3 |\log \rho|^3} \, ds = O(|\log \rho|^{-2}) .
\end{align*}
Letting $\rho\to\infty$, we find that $A_{M}\equiv 0$. This completes the proof.

\section{Proof of Theorem \ref{theo:index}} \label{sec:index}

By \cite[\S 3]{Tysk}, a minimal immersion $M^3\to\RR^4$ with finite total curvature $\int_M |A_M|^3 < \infty$ has finite index. Thus, it suffices to prove that finite index implies finite total curvature. Consider a complete, two-sided, minimal immersion $M^3\to\RR^4$ with finite index. By Proposition \ref{prop:finitely.ends}, $M$ has $k < \infty$ ends. Note that $M$ is oriented and has bounded curvature.

\begin{lemm}\label{lemm:components.ends.exhaustion}
Consider an exhaustion $\Omega_{1}\subset \Omega_{2}\subset \dots \subset M$ by pre-compact regions with smooth boundary so that each component of $M\setminus \Omega_{i}$ is unbounded. Then for $i$ sufficiently large, $\partial\Omega_{i}$ has $k$ components. 
\end{lemm}
\begin{proof}
Discarding finitely many terms, we can assume $M\setminus \Omega_{i}$ has $k$ components for all $i$. If the assertion fails, we can pass to a subsequence so that $\Omega_{i+1}\setminus\Omega_{i}$ has $k$ components and $\partial\Omega_{i}$ has at least two components in a fixed end $E$. This yields a sequence of surfaces $\Sigma_{i}\subset \partial\Omega_{i}$ that form a linearly independent set in $H_{2}(M)$. By Poincar\'e duality, this implies that $H^{1}_{c}(M;\RR)$ is infinite. By \cite[Proposition 2.11]{Carron} and Proposition \ref{prop:MSS}, this further implies that the first $L^{2}$-Betti number is infinite, contradicting Proposition \ref{prop:finitely.ends}. 
\end{proof}

Consider an end $E \subset M$. We can assume that $E$ has smooth boundary. Using the arguments in Proposition \ref{prop:Greens}, we can construct a harmonic function $u \in C^{\infty}(E)$ with finite Dirichlet energy so that $u=1$ on $\partial E$ and $u\to 0$ at infinity. By Lemma \ref{lemm:components.ends.exhaustion}, we can choose $\tau_{0}\in (0,1)$ a regular value of $u$, so that for any other regular value $t\in(0,\tau_{0})$, $\Sigma_{t} : = u^{-1}(t)$ is connected and $\{u>\tau_{0}\}$ is stable. 

The proof of Corollary \ref{coro:F.est.sharp} carries over to this situation to show that
\[
F(t) = \int_{\Sigma_{t}} |\nabla u|^{2}
\]
satisfies $F(t) = O(t^{2})$. We proceed essentially as in the proof of Theorems \ref{theo:bernstein} and \ref{theo:curvature}, and plug the test function $\varphi(u)$ into Proposition \ref{prop:SSY.L3}, where
\[
\varphi(t) = \begin{cases}
0 & t \in [0,\rho^{-2})\\
2 + \tfrac{\log t}{\log \rho }  & t\in [\rho^{-2},\rho^{-1})\\
1 & t \in [\rho^{-1},\tfrac{\tau_0}{2})\\
2-\tfrac{2t}{\tau_0} & t\in [\tfrac{\tau_0}{2},\tau_0].
\end{cases}
\]
Sending $\rho\to \infty$, we conclude that
\[
\int_{E} |A_{M}|^{3} < \infty. 
\]
Since there are finitely many ends (Proposition \ref{prop:finitely.ends}), this yields
\[
\int_{M} |A_{M}|^{3} < \infty,
\]
as desired.

\appendix

\section{Ends of stable (finite index) minimal hypersurfaces}\label{app:ends}

\begin{prop}[\cite{CaoShenZhu}]\label{prop:CSZ}
For $n>2$, if $M^n\to\RR^{n+1}$ is a complete, connected, two-sided, stable minimal immersion then $M$ has only one end. 
\end{prop}

\begin{prop}[\cite{LiWang}]\label{prop:finitely.ends}
For $n>2$, if $M^n\to\RR^{n+1}$ is a complete, connected, two-sided, stable minimal immersion with finite index, then $M$ has finitely many ends. Moreover, the space of $L^2$-harmonic $1$-forms is finite. 
\end{prop}

\section{Harnack and Sobolev inequalities}\label{app:harnack.sobolev}
\begin{prop}[{\cite{Yau:harmonic}, cf.\ \cite[Theorem I.3.1]{SchoenYau:lectures}}]\label{prop:diff.harnack}
Suppose that $(M^n,g)$ is a complete Riemannian manifold and $u$ is a positive harmonic function on $B_r(x)$. If $\Ric \geq - K^2$ on $B_r(x)$ then
\[
|\nabla u|(x) \leq C (r^{-1} + K) u(x)
\]
for $C=C(n)$. 
\end{prop} 

\begin{prop}[\cite{MichaelSimon}]\label{prop:MSS}
For $n>2$, suppose that $M^n\to \RR^{n+1}$ is a complete minimal immersion. Then for any $w \in C^{0,1}_c(M)$, it holds that 
\[
\left(\int_M w^{\frac{2n}{n-2}} \right)^{\frac{n-2}{n}} \leq C \int_M |\nabla w|^2
\]
for $C=C(n)$. 
\end{prop}

\section{Level sets of harmonic functions} \label{app:harmonic}

\begin{prop}[{\cite{HardtSimon,Lin:nodal,CheegerNaberValtorta}}]\label{prop:crit.harm}
Fix a compact subset $K$ of a Riemannian $n$-manifold $(M^n,g)$, if $\Delta u = 0$ is a harmonic function on $(M,g)$ then 
\[
\cH^{n-1}(B_\rho(x) \cap \{u=s\}) \leq C \rho^{n-1}
\]
for any $s\in\RR$, $x \in K$, $\rho \leq \rho_0=\rho_0(M,g,K,u)$ where $C=C(M,g,K,u)$. Furthermore
\[
\dim_\textnormal{Hau} (\{u=s,|\nabla u|=0\}) \leq n-2
\]
for any $s \in \RR$. 
\end{prop}

\section{The Schoen--Simon--Yau $L^3$ estimate}\label{app:L3}
Schoen--Simon--Yau have obtained $L^p$ estimates for the second fundamental form of a stable minimal hypersurface $M^n \to \RR^{n+1}$ in \cite{SSY}. The following $L^3$ estimate follows from their arguments in a straightforward manner but we could not find it in the literature (the estimate from \cite{SSY} is only stated for $p\geq 4$) so we include a proof. 
\begin{prop}\label{prop:SSY.L3}
For $n < 8$, if $M^n \to\RR^{n+1}$ a complete, connected, two-sided, stable minimal immersion then there is $C=C(n)$ so that
\[
\int_{M} |A_{M}|^{3} f^{3} \leq C \int_{M}  |\nabla f|^{3} .
\] 
for any $f\in C^{0,1}_{0}(M)$. 
\end{prop}
\begin{proof}
Set 
\[
a_\delta = (|A_M|^2+\delta)^{\frac 14}.
\]
For $f \in C^\infty_0(M)$ we can take $a_\delta f$ in the stability inequality to find 
\begin{align}
& \int_M |A_M|^2 a_\delta^2 f^2 \nonumber \\
& \leq \int_M |\nabla a_\delta|^2f^2 + f \bangle{\nabla a_\delta^2,\nabla f} + a_\delta^2 |\nabla f|^2 \nonumber\\
& = \int_M \tfrac {1}{16} a_\delta^{-6} |\nabla |A_M|^2|^2 f^2 + \tfrac 12 f a_{\delta}^{-2} \bangle{\nabla |A_{M}|^{2},\nabla f} + a_\delta^2 |\nabla f|^2\label{eq:stab.L3}
\end{align}
Multiplying Simons identity \cite{Simons}
\[
|\nabla A_M|^2 = \tfrac 12 \Delta |A_M|^2 + |A_M|^4
\]
by $a_\delta^{-2}f^2$ and integrating by parts we find 
\begin{align*}
& \int_M a_\delta^{-2} |\nabla A_M|^2 f^2 \\
&= \int_M \tfrac 12 a_\delta^{-2} f^2 \Delta |A_M|^2 + a_\delta^{-2} |A_M|^4 f^2\\
& = \int_M  \tfrac 14   a_\delta^{-6} |\nabla |A_M|^2|^2f^2 - a_\delta^{-2} f \bangle{\nabla  f, \nabla |A_M|^2} + a_\delta^{-2} |A_M|^4 f^2
\end{align*}
so
\begin{align*}
& \int_M a_\delta^{-2} ( |\nabla A_M|^2 f^2 -  \tfrac 14   a_\delta^{-4} |\nabla |A_M|^2|^2f^2)  \\
& = \int_M  - a_{\delta}^{-2} f \bangle{\nabla  f, \nabla |A_{M}|^{2}}  + a_\delta^{-2} |A_M|^4 f^2.
\end{align*}
Note that the improved Kato inequality reads
\[
\tfrac 14 (1+\tfrac 2n) |\nabla |A_M|^{2}| \leq |A_M|^{2} |\nabla A_{M}|^{2} \leq a_{\delta}^{4}|\nabla A_{M}|^{2}
\]
so we can rearrange this into
\begin{align*}
& \int_M \tfrac{1}{2n}  a_\delta^{-6}  |\nabla |A_M|^{2}|  f^2  \leq \int_M   - a_{\delta}^{-2} f \bangle{\nabla  f, \nabla |A_{M}|^{2}}  + a_\delta^{-2} |A_M|^4 f^2.
\end{align*}
Because $a_{\delta}^{-2}\leq |A_{M}|^{-1}$ we find
\begin{align*}
& \int_M \tfrac{1}{2n}  a_\delta^{-6}  |\nabla |A_M|^{2}|  f^2  \leq \int_M   -  a_{\delta}^{-2} f \bangle{\nabla  f, \nabla |A_{M}|^{2}}  + |A_M|^3 f^2,
\end{align*}
so we can combine this with stability to write
\begin{align*}
& \int_M  \tfrac{8-n}{16n}  a_\delta^{-6}  |\nabla |A_M|^{2}|  f^2 \leq \int_M - \tfrac 12 a_{\delta}^{-2} f \bangle{\nabla f,\nabla |A_{M}|^{2}} + a_\delta^2 |\nabla f|^2
\end{align*}
Because $n<8$ we can use Cauchy--Schwartz to find $C=C(n)$ so that 
 \begin{equation}\label{eq:L3.conc1}
 \int_M  a_\delta^{-6}  |\nabla |A_M|^{2}|  f^2 \leq C \int_M  a_\delta^2 |\nabla f|^2
\end{equation}
On the other hand, we can use Cauchy--Schwartz on \eqref{eq:stab.L3} to find 
\begin{equation}\label{eq:L3.conc2}
 \int_M |A_M|^2 a_\delta^2 f^2 \leq \int_M \tfrac {17}{16} a_\delta^{-6} |\nabla |A_M|^2|^2 f^2  + \tfrac{17}{16} a_\delta^2 |\nabla f|^2.
\end{equation}
Combining \eqref{eq:L3.conc1} and \eqref{eq:L3.conc2} and taking the limit as $\delta\to 0$ (using the dominated convergence theorem) we find
\[
\int_{M} |A_{M}|^{3} f^{2} \leq C \int_{M} |A_{M}| |\nabla f|^{2} .
\] 
A standard argument shows that this holds for $f\in C^{0,1}_{c}(M)$. We can then replace $f$ by $f^{\frac 32}$ and using H\"older's inequality to conclude the proof. 
\end{proof}

\section{An extension to almost stable minimal hypersurfaces}\label{app:almost.stable}
\begin{theo}\label{theo:almost.stable}
	Let $\delta_s\in [0,\tfrac 13]$, and $M^3\to \RR^4$ be a complete, connected, two-sided, $\delta_s$-stable minimal immersion, that is,
	\[\int_M (1-\delta_s) |A_M|^2 f^2 \le \int_M |\nabla f|^2\]
	for any $f\in C_0^\infty(M)$. Then $M$ is flat.
\end{theo}

We do not know if the bound for $\delta_s$ assumed here (and in Proposition \ref{prop:almost.stable.tech} below) is sharp. 
The proof of Theorem \ref{theo:almost.stable} closely follows that of Theorem \ref{theo:bernstein}. Note that the condition $\delta_s\le \tfrac 13$ is required to carry out the Cao--Shen--Zhu proof of the one-endedness property for $\delta_s$-stable minimal hypersurfaces (Proposition \ref{prop:CSZ}), cf.\ \cite[Lemma 10.2]{Li:harmonic.lectures}. As such, by following the argument in Section \ref{sec:proofs} we can reduce the proof of Theorem \ref{theo:almost.stable} to the following result:
\begin{prop}\label{prop:almost.stable.tech}
Suppose that $M^3\to\RR^4$ is a complete, connected, simply connected, two-sided, one-ended, minimal immersion $M^3\to\RR^4$ that is $\delta_s$-stable for $\delta_s \in [0,\tfrac 12)$. Then $M$ is flat. 
\end{prop}

Indeed, one may trace through the proof of Corollary \ref{coro:cutoff} to show that the $\delta_s$-stability condition (with $\delta_s < \frac 3 4$) implies:
\begin{align*}
& \limsup_{\ell\searrow 0} \int_{\ell}^{t}   \cA(s)\, ds + t^2 \int_{t}^{1}  s^{-2}  \cA(s)\, ds   \leq O(t)  +  \frac {4}{3-4\delta_s} \int_{t}^{1} t^2 s^{-4} F(s) \, ds. 
\end{align*}
Since Proposition \ref{prop:MW-monotonicity} does not depend on stability, we can thus follow along the proof of Theorem \ref{theo:bernstein} to find
\[
F(t) \leq O(t^2) + \frac{1}{3-4\delta_s} t^3 \int_t^1 s^{-4} F(s) ds.
\]
The proof of $F(t) = O(t^2)$ given in Corollary \ref{coro:F.est.sharp} then carries over as long as $\delta_s < \frac 12$. Finally, the $L^3$-estimate in Proposition \ref{prop:SSY.L3} holds for stable $M^n\to \RR^{n+1}$ as long as $\delta_s < \frac {8-n}{ 8}$. Putting this together, Proposition \ref{prop:almost.stable.tech} (and thus Theorem \ref{theo:almost.stable}) follows. 

%
%


\bibliography{bib}
\bibliographystyle{amsplain}

\end{document}